\documentclass{amsart}  
\usepackage{appendix}
\usepackage{graphicx}
\usepackage[parfill]{parskip} 
\usepackage{amsfonts, amscd, mathrsfs, amsmath, amssymb, amsthm}
\usepackage{url}
\usepackage{hyperref} 
\hypersetup{backref,pdfpagemode=FullScreen,colorlinks=true}
\usepackage{tikz-cd}
\usepackage{mathtools}
\usepackage{color, verbatim}
\usepackage{bm}
\usepackage[hmargin=3cm,vmargin=3cm]{geometry}
\usepackage{indentfirst}

\numberwithin{equation}{section}

\newtheorem{theorem}[equation]{Theorem} 
\newtheorem{proposition}[equation]{Proposition}
 
\newtheorem{lemma}[equation]{Lemma} 
 
\newtheorem{corollary}[equation]{Corollary} 
\newtheorem{remark}[equation]{Remark}

\newtheorem{definition}[equation]{Definition}
\theoremstyle{definition}

\newtheorem{notation}{Notation}

\theoremstyle{remark}

\newtheorem{example}[equation]{Example}
\newtheorem{question}{Question}

\DeclareMathOperator {\length} {length}

\DeclareMathOperator{\image}{\mathrm{image}}
\DeclareMathOperator{\domain}{\mathrm{domain}}

\begin{document}
\title {Incompleteness theorems via Turing category} 
\author{Yasha Savelyev} 
\href{http://yashamon.github.io/web2/papers/turingcategory.pdf} {Direct link to author's version}
\address{University of Colima, Department of Sciences, CUICBAS}
\email {yasha.savelyev@gmail.com}
\begin{abstract}
We give a reframing of
G\"odel's first and second incompleteness theorems that
applies even to some undefinable theories of arithmetic.  The usual Hilbert–Bernays provability
conditions and the ``diagonal
lemma'' are absent, replaced by a more
direct diagonalization argument, from first principles, based
in category theory and in a sense analogous to Cantor's original
argument.  To this end, we categorify the theory G\"odel encodings, which might be of
independent interest. In our setup, the G\"odel sentence is 
computable explicitly by construction even for $\Sigma ^{0} _{2}$
theories (likely extending to $\Sigma ^{0} _{n}$).  In an
appendix, we study the relationship of our reframed second
incompleteness theorem with arguments of Penrose.
\end{abstract}
\maketitle
\section {Introduction} \label{section:intro}
Ordinarily, G\"odel incompleteness theorems are formulated
for definable arithmetic theories. In fact, the original
incompleteness theorems applied to computably enumerable
theories of arithmetic, but the latter can be generalized in
some ways to $\Sigma _{n}$ theories, see for instance
~\cite{cite_KikuchiGeneralizationsGodel}. One of the goals
here is to extend first and second incompleteness theorems
to certain undefinable theories of arithmetic. Another goal
is to rework the incompleteness theorems from a more set
theoretic/categorical point of view, utilizing a certain category of
G\"odel encodings. The extension to some undefinable
theories is then one proof of concept, for this
recasting of the incompleteness theorems.

We will work from first principles, and within the
meta-theory $ZFC$, that is all theorems are theorems of
$ZFC $ (and mostly just $ZF$). Our setup will be partly language agnostic, we
work with a general $\mathcal{L} $-theory, interpreting
arithmetic but our notion of interpretation is very weak. 
Partly for this reason various standard tools like
Hilbert-Bernays provability conditions and the ``diagonal
lemma''  will not appear, as translating
them to our context is difficult.  
Instead, the
diagonalization argument that we use is more direct,
in a sense elementary, and more analogous to the
set theoretic diagonalization argument of Cantor, as well as to Turing's original work on the halting
problem. To this end we also categorify the
theory of G\"odel encodings, obtaining what we call the
Turing category. This might be of independent interest.

If we simplify some assumptions \footnote{For example
restrict to c.e. theories of arithmetic as in the original
incompleteness theorems.},  our
argument drastically simplifies, and may be of interest purely as
another, from first principles of Turing machines approach to the classical incompleteness theorems. 

The primary motivation and many of the choices we make, most notably
the choice to work with stably c.e. theories (see following
section), are
based around the problem of formalizing a certain version of
an argument of Roger Penrose on potential non-computability in
physics. The reader may see Appendix \ref{section:Penrose},
however many terms there are defined throughout the paper.

\subsection {Definition of stably c.e. theories and
statements of results}
Let $\mathcal{L} $
be a first order language, in particular a formal language
in a countable alphabet including the
symbols of first order logic, with syntax of first
order logic, such that elements of $\mathcal{L} $ are
sentences (in the logical sense). 

We now quickly introduce stable c.e. theories. In what
follows, a map is always a partial map, unless it is specified to be total.
Suppose we are given a map $$M: \mathbb{N} \to \mathcal{L}  \times \{\pm\},
$$ for $\{\pm\}$ denoting a set with two elements $+,-$. 
\begin{definition} \label{def:main} 
We say that  \textbf{\emph{$\alpha \in \mathcal{L}$ is $M$-stable}} if there is an $m$ with $M (m) = (\alpha,+)$ s.t. there is no $n>m $ with $M (n) =  (\alpha,-)$. Let $M ^{s} \subset 
\mathcal{L}$ denote the set of $M$-stable $\alpha$,  
called \textbf{\emph{the stabilization of $M$}}. 
\end{definition}
\begin{remark} \label{remark_intro}
For an informal motivation of how such an $M$ may appear in
practice consider the following. With $\mathbb{N} $ playing
the role of time, $M$ might be a mathematician producing
sentences of arithmetic that it believes to be true, at each
moment $n \in \mathbb{N}$. But $M$ is also allowed to correct
itself in the following sense. 
\begin{itemize}
	\item $M  (n) = 
(\alpha,+),$ only if at the moment $n$ $M$ decides 
that $\alpha$ is true.
\item $M (m) = 
(\alpha,-),$ only if at the moment $m$, $M$ no 
longer asserts that $\alpha$ is true, either 
because at this moment $M$ is no longer able to decide 
$\alpha$, or because it has decided it to be 
false. 
\end{itemize}
\end{remark}
\begin{definition} \label{def:stablycomputableintro} 
A subset $S \subset \mathcal{L} $, is called 
\textbf{\emph{stably computably enumerable}} or 
stably c.e., if there is a computable map
(see Definition \ref{def:computable}) $T: \mathbb{N} \to 
\mathcal{L} \times \{\pm \} $ so that $S = T ^{s} 
$. In this case we also say that $T$ \textbf{\emph{stably
enumerates}} $S$.  We say that $T \in \mathcal{T} $ \textbf{\emph{stably
computes}} $ M: \mathbb{N} \to \mathcal{L} \times \{\pm\},$
if it computes some $N: \mathbb{N} \to  \mathcal{L}  \times \{\pm \}$, s.t. $M ^{s} = N ^{s}$.
\end{definition} 
It is fairly immediate that a stably c.e. $S$ is $\Sigma
_{2} = \Sigma ^{0} _{2}$  definable. The converse is also true, every $\Sigma
_{2}$ definable set $S \subset \mathcal{L} $ is stably c.e.. To prove this 
we may build on Example \ref{remark:dio}, to construct an
oracle and then use the theorems of Post, (see
\cite{cite_Soare}), relating the arithmetic hierarchy with
the theory of Turing degrees. We omit the details as this
will not be essentially used, and is a well understood idea.


Let $\mathcal{A} $ denote the first order language of arithmetic with non-logical symbols $\{0, +, \times, s, < \})$. 
In the following we need a notion of a $n$-translation from
the language of arithmetic to another abstract language,  as well as the notion of
$n$-consistency and strong consistency. The preliminaries
for this are given in Section \ref{sec_firstorder}.


Note that because our notion of a translation is weak, we
cannot reduce the following to the second incompleteness
theorems for definable arithmetic theories. This is because the
corresponding arithmetic theory, that is the theory $F _{i,
\mathcal{A}}$ in the notation of Section
\ref{sec_firstorder}, will generally not be arithmetically definable. 

\begin{theorem} \label{thm:corollarySecondIncompleteness}
Let $F$ be a theory in any language $\mathcal{L}$, such that $F
\vdash ^{i} ZFC$ for a  2-translation $i$. Then if $F$ is strongly consistent:
\begin{equation}
    \label{eq:F1cons}
(F \nvdash ^{i} \text{$F$ is 1-consistent}) \lor
(F \nvdash ^{i} F \text{ is stably c.e., i.e. is $\Sigma _{2}$}). 
\end{equation}
\end{theorem}

The above is based on the following results.
%
The next theorem in particular tells us that the G\"odel
sentence is ``computable'', even in the setup of abstract
$\mathcal{L} $-theories above.
\begin{theorem} \label{thm_mainintro}  
Let $\mathcal{L} $ be any fixed first order language, and
let $i: \mathcal{A} \to \mathcal{L} $ be some 2-translation. Then there
is a total computable map $$\mathcal{G}: \mathcal{T} \to
\mathcal{A},$$ (depending only on $i$) where the domain $\mathcal{T}$ is the set of Turing machines $\mathbb{N} \to \mathbb{N} $, and where $\mathcal{G} $ satisfies the following.  Suppose that $T$ \emph{stably computes}  $M:
\mathbb{N} \to \mathcal{L} \times \{\pm \}$. Let $F= M
^{s}$, then we have:
 \begin{enumerate}
	\item \text{$F \nvdash ^{i} \mathcal{G}
	(T)$}  if $F$ is 1-consistent. \label{part_1}
	\item $F \nvdash ^{i} \neg \mathcal{G}
	(T)$ if $F$ is 2-consistent.
	\item 
\begin{equation} \label{eq:secondincompletness}
(\text{$F$ is $1$-consistent)} \implies \mathcal{G} (T).
\end{equation}

\label{part_3}
\end{enumerate}
Furthermore, the Turing machine computing
$\mathcal{G} $, can itself be given constructively.
\end{theorem}
As a corollary we get a more basic form of Theorem \ref{thm:corollarySecondIncompleteness}:  
\begin{corollary} \label{cor_secondinc} 
For $F$ a theory in any language $\mathcal{L}$ such that $F
\vdash ^{i} ZFC$, where $i$ is a 2-translation: 
\begin{equation*}
\forall T \in \mathcal{T}  \; ((\text{$F$ is 1-consistent})
\land (F  \vdash ^{i} \text{$T$ stably enumerates
$F$}) \implies (F \nvdash ^{i}
\text{$F$ is 1-consistent})).   
\end{equation*}
\end{corollary}

\subsection{Generalizations to $\Sigma _{n}$} 
There are natural candidates for how to generalize the above. We may replace $M: \mathbb{N} \to \mathcal{L} \times \{\pm \} $ by $M: \mathbb{N} ^{n} \to \mathcal{L} 
\times \{\pm \} $, using this we can define a notion of 
$n$-stable computability, specializing to stable 
computability for $n=1$. The above theorems should 
generalize to this setting of $n$-stable computability. In
terms of arithmetic complexity this should be exactly the
class $\Sigma ^{0} _{n+1}$. We leave this for future developments.  
\section {Some preliminaries} \label{sec:prelims}
\subsection{Abstractly encoded sets and the Turing category} \label{sec:abstractencodedsets}
The material of this section will be used in the 
main argument.  
The approach here is in essence the \emph{standard} approach,
but usually encoding systems are implicit and fixed, as
ordinarily one is working with some concrete sets and
concrete encodings. We need to work with abstract sets (in
particular all abstract first order languages), and we need
to formalize the properties of encodings in such a way that
showing maps are computable can be done axiomatically, as
the maps we construct are fairly complex.  For this, it is very natural to use the language of category theory.

\begin{definition}\label{def_}
	An \textbf{\emph{encoding map}} of a set $A$ is an injective
total set map $e: A \to \mathbb{N} $, such that the set $A _{e}
= e (A)$ is computable (recursive).
\end{definition}
 Here, as is standard, a 
set $S \subset \mathbb{N} $ is called 
\emph{computable} if both $S$ and its complement 
are computably enumerable, with $S$ called 
\emph{computably enumerable} if there is a 
computable partial function $\mathbb{N} \to \mathbb{N}$ with
range $S$.   

We extend the collection of encodings to a structure of a category. 
\begin{definition} We denote by $\mathcal{T}$ the 
set of all Turing machines $T: \mathbb{N} \to \mathbb{N} $.
We write $*T (n)$ for the computation sequence of the
Turing machine $T$ with input $n$. 
As usual, for $T \in \mathcal{T} $, $T$ also denotes the
underlying partial function with $T (n) = m$ if $*T (n) $ halts with output $m$, and undefined otherwise. 
\end{definition}
In what follows, a map is a partial map, unless we specify
that it is total. Define a large (the set of objects
is a proper class) arrow type category
${\mathcal{S}}$ whose objects are encoding maps $e
_{A}: A \to \mathbb{N}$.  
Explicitly, objects $\operatorname {obj} \mathcal{S}  $ of
$ \mathcal{S}$ consists of pairs $(A, e _{A}) $ where $A$ is
a set, and $e _{A}: A \to \mathbb{N}$ an encoding map.

We now describe the morphisms of $\mathcal{S} $.
\begin{definition} \label{def:computable} For $(N, e _{N})$,
$(M, e _{M})$ in $\operatorname {obj}  \mathcal{S} $,
a morphism is a map $f$ s.t. there is a commutative diagram:
\begin{equation*}
\begin{tikzcd}
N \ar[r, "f"] \ar [d, "e _{N}"] &  M \ar [d,"e _{M}"] \\
\mathbb{N}  \ar [r, "T"]   & \mathbb{N},
\end{tikzcd}
\end{equation*}
for some $T \in \mathcal{T}$. 
\end{definition}
To simplify notation, we may omit specifying the encoding map
for a given object of $\mathcal{S}$, keeping track of it
implicitly. With this in mind, for $N,M \in \mathcal{S} $ we
say that $T$ \textbf{\emph{computes}} a map $f: N \to M$ if $T$ fits into
a commutative diagram as above.
We say that $f: N \to M$ is \textbf{\emph{computable}}  if there exists a $T \in \mathcal{T} $ which computes $f$.

So in the terms above, the set $hom _{\mathcal{S} } ((N, e _{N}), (M, e _{M}))$ is the set of computable maps $f: N \to M$.
In what follows, if not specified, the encoding of
$\mathbb{N} $ is the $id: \mathbb{N} \to \mathbb{N} $. \begin{proposition} \label{prp_propertiesOfS}
The category $\mathcal{S} $, called the \textbf{\emph{Turing category}}, satisfies the following:
\begin{enumerate}
\item \label{axiom:product} The category $\mathcal{S} $ has
finite products. (In fact it has all finite limits and
colimits, but we will not need this.) This entails the following. If $(A, e _{A}), (B, e _{B}) \in 
\mathcal{S}$ then there is a distinguished encoding map $e
_{A \times B}: A \times B \to \mathbb{N} $, called the
\textbf{\emph{product encoding map}}  s.t.:
\begin{enumerate}
\item The projection maps ${pr} ^{A} : A \times B \to A$, ${pr}^B: A \times B \to B$ are computable. Similarly for $pr ^{B} $.   
\item \label{axiom:slice} If $f: A \to B$ 
is computable, and 
$g:A \to C$ is computable then  $A \to B \times C, $ $a 
\mapsto (f (a), g (a))$ is computable.  
\item If $f: A \to B$, 
$g: C \to D$ are computable 
then the map $A \times B \to C \times D $, $(a,b) 
\mapsto (f (a), g (b) ) $ is computable.	
\end{enumerate}
From now on, the encoding map of $A \times B$ is taken to be
the product encoding map (with respect to some
possibly implicit encoding maps of $A, B$).
\item \label{axiom:universal} The set  
$\mathcal{T}$ has a distinguished encoding $e
_{\mathcal{T}}$ s.t. the following holds.  
Define $$U: 
\mathcal{T} \times \mathbb{N}  \to \mathbb{N},$$ $$U 
(T, \Sigma ):= \begin{cases}
	T (\Sigma), &\text{ if $*T(\Sigma)$ halts} \\
\text{undefined}, &\text{otherwise}.
\end{cases}
$$ 
Then $U$ is computable.
(The Turing machine computing $U$ may be called the ``universal Turing machine''.)  From now $\mathcal{T} $ as an object $\mathcal{S} $ is taken with respect to this distinguished encoding.
\item \label{axiom:split}  
Let $A,B,C \in \mathcal{S} $, and suppose that $f: A \times B \to C$ is computable. Let $f ^{a}: B \to C $ be the map $f ^{a}
(b) = f (a,b) $. Then there is a computable map $$s: A \to 
\mathcal{T}$$ so that for each $a$ $s 
(a)$ computes $f ^{a}$.
\item \label{axiom:utility}  For a set $A$ let $$L ({A} ) := \bigcup _{n 
\in \mathbb{N}} Maps (\{0, \ldots, n\}, {A}), $$ where $Maps (\{0, \ldots, n\}, {A})$ denotes the set of 
total maps. If $(A, e _{A}) \in \mathcal{S} $ then
there is a distinguished encoding $e _{L (A)}$ of $L ({A} )$ s.t.:
\begin{enumerate}
\item The length function $$\length: L ({A} ) \to \mathbb{N},$$ 
is computable, where for $l \in L 
(\mathcal{A} ) $, $l: \{0, \ldots n \} \to A$,  $\length (l)=n$.  
\item Define $$P: L ({A}) \times \mathbb{N}  \to 
{A},$$ 
$$P (l,i): = \begin{cases}
	l (i), & \text{ if $0 \leq i \leq \length (l)$ }\\
 \text{undefined}, & \text{for $i> 
\length (l) $}.  	
\end{cases} $$
Then $P$ is computable.
\item \label{axiom:lists} For $A,B \in \mathcal{S} $ and $f: A \to L 
(B) $ a partial map, suppose that:
\begin{itemize}
\item The partial map $ A \times \mathbb{N}  \to B$, $(a,n)
\mapsto P (f (a),n) $  is computable.
\item The partial map $A \to \mathbb{N} $, $a 
\mapsto \length (f (a))  $ is computable.
\end{itemize}
Then $f$ is computable. From now, given $A \in \mathcal{S} $ the encoding of $L (A)$ is
assumed to be such a distinguished encoding, called the
\textbf{\emph{list encoding map}}.
\end{enumerate} 
\item \label{axiom_L} Let $\mathcal{L} $ be a first order
language, then there is a distinguished encoding of
$\mathcal{L} $ with the following property. 
There is a total computable map:
   \begin{equation*}
      \Phi: L (\mathcal{L} ) \times \mathbb{N} \to
			\mathcal{L}, 
   \end{equation*}
  s.t. for each $l \in L 
  (\mathcal{L}) $, $\Phi (\{l\} \times  \mathbb{N}  
  ) $ is the deductive closure of the theory $F _{l}
	= \image l$. From now such an $\mathcal{L} $ is assumed to
	have such a distinguished encoding.
\end{enumerate}
\end{proposition}

\begin{lemma}
   \label{lemma:lists}
If $f: A \to B$ is computable 
then the map $L(f): L (A) \to L (B) $, $$l \mapsto 
\begin{cases}
   i \mapsto f(l (i)), &\text{ if $f(l (i)) $ is 
   defined for all $0 \leq i \leq \length (l) $  } \\
   undefined, & \mbox {otherwise}, 
\end{cases}
$$ is computable. Also, 
the map $LU: \mathcal{T} \times L (\mathcal{U}) 
\to L (\mathcal{U}) $, $$l \mapsto  
\begin{cases}
i \mapsto U (T, (l (i))), &\text{ if $U (T, (l (i))) $  is 
   defined for all $0 \leq i \leq \length (l) $}\\
undefined, &\text { otherwise}
\end{cases}  $$  is computable.
\end{lemma}
\begin{proof}
This is just a straightforward application of  
the proposition and part \ref{axiom:utility} in particular. We leave the details as an 
exercise.
\end{proof}


\begin{proof} [Proof of Proposition \ref{prp_propertiesOfS}]
As this is just an elaboration on classical theory we only
sketch the proof. To prove the first part,  given $(A,
e _{A}),(B, e _{B})$ the encoding map $e _{A \times
B}: A \times B \to \mathbb{N} $ can be taken to be the map $(a, b) \mapsto 2 ^{e _{A} (a)} \cdot {3} ^{e _{B} (b)}$. The needed properties readily follow.

The second part is just the classical story of the universal
Turing machine. 

The third part corresponds to the
``s-m-n theorem'' Soare~\cite[Theorem 1.5.5]{cite_Soare},
which works as follows. Given a classical 2-input Turing machine $$T: 
\mathbb{N} \times \mathbb{N} \to \mathbb{N}, $$ 
there is a Turing machine $s _{T}: \mathbb{N} \to 
\mathbb{N} $ s.t. for each $m$ $s _{T} (m) 
$ is the Turing-G\"odel encoding natural, of a Turing 
machine computing the map $f ^{m}$: $n \mapsto T 
(m, n) $.

For the fourth part we can just explicitly
construct the needed encoding, by setting $e _{L (A)}$ to be
the map $$l \mapsto 2 ^{e _{A}
(l (0))} \cdot \ldots \cdot p _{n} ^{e _{A} (l (n))},$$
where $l: \{0, \ldots, n\} \to A $ and $p _{n}$ is the $n$'th prime.

The last part follows by basic theory of first order logic
and the previous parts.
\end{proof}
\begin{remark} \label{rem_}
The above properties suffice for our purposes. As mentioned
$\mathcal{S} $ in fact has all finite limits and colimits,
this is proved analogously. For example the sum (coproduct)
of $(A, e _{A})$ and $(B, e _{B})$, can be given by $(A
\sqcup B, e _{A
\sqcup B}) $ where 
$$e _{A \sqcup B} (x) = \begin{cases}
	2 ^{e _{A} (x)}, &\text{ if } x \in A\\
	3 ^{e _{B} (x)}, & \text{ if } x \in B. 
\end{cases}
$$
$\mathcal{S} $ having finite products and
sums is in part what it makes it possible to have a computer programming language with algebraic data types, e.g. Haskell.
Haskell also has more general,
finite ``colimit, limit'' data types. 
\end{remark}

\subsection {Some preliminaries on first order theories} \label{sec_firstorder}
Let $\mathcal{L} $ be a first order language. An $\mathcal{L}$-theory $F$ is a subset $F \subset \mathcal{L} $. $F$ will be called \emph{deductively complete}  if it is closed under inference, that is if $F \vdash \alpha $ then $\alpha \in F$. Denote by $\overline{F} $ the deductive closure of $F$. 
$$\overline{F} = \cap \{H \subset \mathcal{L} \,|\,
H \text{ is deductively complete and } F \subset H \}.
$$ 
We will need a notion of one first order theory interpreting
another first order theory, possibly in a different
language. 

\begin{definition}\label{def_translation} Given first order languages
$\mathcal{L}, \mathcal{L}'$, a \textbf{\emph{translation}}
is a total set embedding $i: \mathcal{L}' \to \mathcal{L}$ such
that 
\begin{enumerate}
	\item $i$ preserves the logical operators $\land, \lor,
	\neg$. (For example, $i (\neg \alpha \land  \beta
	) = \neg i (\alpha ) \land i (\beta) $). 
	\item
	\begin{equation*} (\forall \alpha \in \mathcal{L}' \;
	\forall S \subset \mathcal{L}') \; S \vdash \alpha
	\implies i (S) \vdash i (\alpha ). 
	\end{equation*}
\end{enumerate}
\end{definition}

Here is one well understood example.
\begin{example} \label{exm_}
If $\mathcal{Z}$ denotes the first order language of set
theory and $\mathcal{A}$ the first order language of
arithmetic as in the introduction, then there is a translation
$i _{\mathcal{A}, \mathcal{Z}}: \mathcal{A}  \to
\mathcal{Z}$.  This map assigns to numerals in the
language of arithmetic the Von Neumann naturals, e.g. we
assign $\emptyset $ to $0$.   In this example, assuming
standard encodings of $\mathcal{A}, \mathcal{Z} $, $i$ is
computable.
\end{example}
If $i$ as above is computable we call $i$
a \textbf{\emph{computable translation map}}. The following,
weakens this notion.
\begin{definition}\label{def_good}
 For
a translation $i: \mathcal{A} \to \mathcal{L} $, if the 
restriction of $i$ to the subset of $\Sigma ^{0} _{n}$
formulas is computable we say that $i$ is
a \textbf{\emph{$n$-translation}}. Likewise, for
a translation $j: \mathcal{Z} \to \mathcal{L} $  we say it is
a $n$-translation, if the induced translation $j \circ i _{\mathcal{A},
\mathcal{Z} }: \mathcal{A} \to
\mathcal{L} $, where $i _{\mathcal{A}, \mathcal{Z} }$ is as
above, is a $n$-translation.
\end{definition}

\begin{definition}
\label{not_induced} Given a translation $i: \mathcal{L}'
\to \mathcal{L} $ and an $\mathcal{L} $-theory $F$, we set $$F
_{i, \mathcal{L}'} = i ^{-1} (\overline{F}) \subset
\mathcal{L}'. $$ To paraphrase, this is set of $\mathcal{L}'
$-sentences proved by $F$ under the given translation.
\end{definition}

The following is immediate from the definitions.
\begin{lemma} \label{lem_} For $\mathcal{L}, \mathcal{L}',
i, F$ as above $F _{i,\mathcal{L}'}$ is a deductively closed
theory.
\end{lemma}
\begin{definition}\label{def_interprets}
For a given translation $i: \mathcal{L}' \to  \mathcal{L}
$, given an $\mathcal{L} $-theory $F$ and a sentence $\alpha
\in \mathcal{L}'$,  we write $F \vdash ^{i}  \alpha$ if
\begin{equation*}
F _{i,\mathcal{L}'} \vdash \alpha. 
\end{equation*}
Likewise, we write $F \vdash ^{i}  F'$  if
\begin{equation*}
F _{i,\mathcal{L}'} \vdash F'. 
\end{equation*}
Whenever, there exists an $i$ s.t. $F \vdash ^{i}  F'$, we say that $F$ \textbf{\emph{interprets}} $F'$. 
\end{definition}
The following is also immediate from definitions.
\begin{lemma} \label{lem_inducedinterpretation} Given first
order languages $\mathcal{L} _{0}, \mathcal{L} _{1},
\mathcal{L}  _{2} $ and an $\mathcal{L} _{2}$-theory $F
_{2}$. If $F _{2}$ interprets $ F _{1}$ and $F _{1}$
interprets  $F _{0}$ then $F _{2}$ interprets $F _{0}$.
\end{lemma}
Let $\mathcal{F}_{0}$ denote the set of $\Sigma ^{0}
_{0}$ formulas of arithmetic $\mathcal{A} $ with one
free variable. 

\begin{definition} [cf. ~\cite{cite_fefermanconsistency}] \label{def:1-consistent}
Given a first order theory $F$ in any language $\mathcal{L} $, we say that $F$
\textbf{\emph{$1$-consistent relative to the translation
$i$}}  if: 
\begin{enumerate}
	\item $F \vdash ^{i} Q$ as in Definition \ref{def_interprets}.
\item For any formula $\phi \in \mathcal{F} _{0} $  the following holds:
$$F \vdash ^{i}  \exists m \; \phi (m)) \implies 
(\exists m \; F \nvdash ^{i} \neg \phi (m)).
$$  
\end{enumerate}
We say 
that $F$ is \textbf{\emph{$2$-consistent relative to the
translation $i$}} if the 
same holds for $\Pi _{1} ^{0} $ formulas $\phi$ with 
one free variable, more specifically formulas $\phi=\forall
n \, g (m,n)$, with $g$ $\Sigma ^{0} _{0}$.
\end{definition}
From now on $n$-consistency is always with respect to some
implicit translation $i$, that should be clear from the
context, and so may not be denoted. 

\begin{definition} \label{def_stronglyconsistent}
Suppose we are given a theory $F$ in some language
$\mathcal{L} $, such that $F \vdash ^{i} ZFC$.  Then we say
that it is \textbf{\emph{strongly consistent}} if there is
a `standard model' $M$ for $F _{i, \mathcal{Z}}$. More
specifically, we suppose that $M$ is
a substructure of $V _{\kappa}$ for $V _{\kappa}$ some
stage in the Von Neumann hierarchy.  
\end{definition}

\section {Stable computability and decision maps} 
\label{section:fundamentalSoundness} In this 
section, general sets, often denoted as $B$, are 
intended to be objects of $\mathcal{S} $ with an implicit
encoding map (sometimes made explicit). All maps are partial maps, unless specified otherwise.
The set $\{\pm \}$ is always understood to be with the fixed
encoding map $e _{\{\pm \}} (-) = 0$, $e _{\{\pm \}} (+)
= 1$.
\begin{definition} \label{def:Mstable}
Given a map: \begin{equation*}
M: \mathbb{N} \to B \times \{\pm\},
\end{equation*}
We say that \textbf{\emph{$b \in B$ is $M$-stable}} if 
there is an $m$ with $M (m) = (b,+)$ 
and there is no $n>m$ with $M (n) =  (b,-)$.
\end{definition}
\begin{definition} \label{def:stabilization} Given a map $$
M: \mathbb{N} \to B \times \{\pm\},$$ we define 
$$M ^{s} \subset  B $$ to be the set of all the $M$-stable $b$. We 
call this the \textbf{\emph{stabilization of $M$}}. When $M$
is computable, that is furnishes a morphism in $\mathcal{S}
$,  we say that $S \subset B$ is 
\textbf{\emph{stably c.e.}} if $S = M ^{s} $. 
We say that $T \in \mathcal{T} $ \textbf{\emph{stably
computes}} $ M: \mathbb{N} \to B \times \{\pm\},$ if it
computes $N: \mathbb{N} \to  B \times \{\pm \}$, s.t. $M ^{s}
= N ^{s}$.
\end{definition}
In general $M ^{s} $ may not be computable even if $M$ is computable. 
Explicit examples of this sort can be readily 
constructed as shown in the following.
\begin{example} \label{remark:dio} 
Let $Pol$ denote the set of all Diophantine 
polynomials, with a distinguished encoding map whose properties will
be specified shortly. 
We can construct a total computable map $$A: 
\mathbb{N} \to Pol \times \{\pm \} $$ whose 
stabilization consists of all Diophantine 
(integer coefficients) polynomials with no integer 
roots. 

Fix a distinguished encodings of $Pol, \mathbb{Z} $  
so that the map $$E: \mathbb{Z} 
\times Pol \to \mathbb{Z}, \quad (n,p) \mapsto p 
(n)$$ is computable.  Let $$Z: \mathbb{N} \to Pol, \quad N:
\mathbb{N} \to \mathbb{Z}$$ be any total bijective computable maps. 

In what follows, for each $n \in \mathbb{N}$,  $A 
_{n} \in L (Pol \times \{\pm\}) $. $\cup$ will be 
here and elsewhere in the paper the natural list 
union operation. More specifically, if $$l _{1}: 
\{0, \ldots, n\} \to B, \quad l _{2}: \{0, \ldots, m\} 
\to B$$ are two lists then $l _{1} 
\cup l _{2}$ is defined by:
\begin{equation} \label{eq:listunion}
   l _{1} \cup l _{2} (i) = 
   \begin{cases}
      l _{1} (i) , &\text{ if $i \in \{0, \ldots, 
      n\}$  }\\
      l _{2} (i -n -1) , &\text{ if $i \in \{n+1, \ldots, 
      n+m+1 \}$ }.
   \end{cases}
\end{equation}
If $B \in \mathcal{S} $, it is easy to see that $$\cup: L (B) \times L (B) \to L (B) , \quad (l,l') \mapsto l \cup l' $$ is 
computable, given that we are using the list encoding map
for $L (A)$, as in part \ref{axiom:lists} of Proposition
\ref{prp_propertiesOfS}.

For $n \in \mathbb{N} $ define $A _{n}$ recursively by: $A _{0}:=\emptyset,$ 
\begin{align*} A _{n+1}& := A _{n}  \cup  \bigcup _{m=0} ^{n}  (Z (m), d ^{n}   (Z (m))), \\
   & \mbox{ where  $d ^{n} (p) = +$ if none of 
   $\{N(0), \ldots, N(n)\}$ are roots of  $p$, $d 
   ^{n}  (p)= -$ otherwise}.
 \end{align*}
Define $A (n):=A _{n+1} (n) $. Note that $$(\forall n \in \mathbb{N}) \; A _{n+1}| 
_{\domain A _{n}} = A _{n},  \text{ and }  \length (A 
_{n+1})  > \length (A _{n}), $$ so that with this definition $A 
(\mathbb{N}) = \cup _{n \in \mathbb{N} } \image (A 
_{n}).$ 

Since $E$ is computable, utilizing the
recursive program above and Proposition
\ref{prp_propertiesOfS}, it can be
readily verified that the map $A$ is computable. Moreover, by 
construction the stabilization $A ^{s}$ 
consists of all Diophantine polynomials that have no integer roots.  
\end{example} 
\subsection {Decision maps} \label{sec_decisionmaps}
By a \textbf{\emph{decision map}}, we mean a 
map of the form:
$$D: B \times \mathbb{N}  \to \{\pm\}.
$$ 
This kind of maps will play in the 
incompleteness theorems, and we now 
develop some of their theory.

\begin{definition}\label{def:calDB}
Let $B \in \mathcal{S} $, define
$\mathcal{D}_{B}$ to be the set of $T \in  \mathcal{T}$ s.t.
exists $T': B \times \mathbb{N} \to \{\pm\}  $, and 
a commutative diagram:
\begin{equation*}
\begin{tikzcd}
B \times \mathbb{N}  \ar[r, "T'"] \ar [d, "e _{B \times \mathbb{N}}"] &  \{\pm \} \ar [d,"e _{\{\pm \}}"] \\ 
\mathbb{N} \ar [r, "T"]    &  \mathbb{N}.
\end{tikzcd}
\end{equation*}
\end{definition}
More concretely, this is the set of $T$ s.t.:  $$(\forall n \in 
\image e_{B \times \mathbb{N}} \subset \mathbb{N}) \; (T (n) \in \image e _{B \times \{\pm \}} \text{ or $T (n)$ is
undefined}.)$$

As $e _{\{\pm \}}$ is injective, $T'$ above is uniquely
determined \emph{if}  it exists. From now on, 
for $T \in \mathcal{D} _{B}$, 
when we write $T'$ it is meant to be of the form 
above.

First we will explain  one construction of elements of $\mathcal{D} 
_{B}$, from  Turing machines of the following form.

\begin{definition}\label{def:T_B}
Let $B \in \mathcal{S} $. Define
$\mathcal{T}_{B}$ to be the set of $T \in  \mathcal{T}$ s.t.
exists $T': \mathbb{N} \to B \times \{\pm \}$, and 
a commutative diagram:
\begin{equation*}
\begin{tikzcd}
\mathbb{N}  \ar[r, "T'"] \ar [d, "id"]
&  B \times \{\pm \} \ar [d,"e _{B \times \{\pm \}}"] \\ 
\mathbb{N} \ar [r, "T"]    &  \mathbb{N}.
\end{tikzcd}
\end{equation*}
\end{definition}

From now on, given $T \in \mathcal{T} 
_{{B}}$, if we write $T'$ then it is will be 
assumed to be of the form above. As before, it is uniquely
determined when exists.
\begin{lemma} \label{lem_inverse} For $(A,e _{A}) \in
\mathcal{S} $ define the map:
\begin{equation} \label{eq_inverse}
e _{A} ^{-1}: \mathbb{N} \to A,
\end{equation}
by $$e _{A} ^{-1} (n) = \begin{cases}
	e ^{-1} _{A} (n), &\text{ if $n \in \image e _{A}$} \\
	\text{undefined}, & \text{ otherwise}. 
\end{cases}$$
Then $e _{A} ^{-1}$ is computable.
\end{lemma}
\begin{proof} [Proof]
Let $T \in \mathcal{T}$ compute the map $\mathbb{N} \to
\mathbb{N} $ defined by $$n
\mapsto   \begin{cases}
	n, &\text{ if $n \in \image e _{A}$} \\
	\text{undefined}, & \text{ otherwise}. 
\end{cases}$$
As $\image e _{A}$ is decidable $T$ does exist.

Then
clearly we have a commutative diagram:
\begin{equation*}
\begin{tikzcd}
\mathbb{N} \ar[r, "e _{A} ^{-1}"] \ar [d, "id"] &  A \ar [d,"e _{A}"] \\
\mathbb{N}  \ar [r, "T"]   & \mathbb{N},
\end{tikzcd}
\end{equation*}
so that $e _{A} ^{-1}$ is computable.
\end{proof}

\begin{lemma} \label{lemma:TuringK} Let $(B, e _{B}) \in \mathcal{S} $.
   There is a computable total map $$K = K _{(B, e _{B})}: 
   \mathcal{T} \to \mathcal{T},$$ with the 
   properties:
   \begin{enumerate}
      \item For each $T$, $K  (T) \in 
      \mathcal{T} _{B }$. 
      \item If $T \in \mathcal{T} _{B }$ then  $K 
      (T) $ and $T$ compute the same maps $ 
      \mathbb{N} \to B \times \{\pm \} $.   
   \end{enumerate}
\end{lemma}
\begin{proof}
Let $G: \mathcal{T} \times \mathbb{N} \to B \times \{\pm \} $  be the composition of the sequence of maps $$
\mathcal{T} \times \mathbb{N} \xrightarrow {U} \mathbb {N} 
\xrightarrow{e _{B \times \{\pm \} } ^{-1}}
B 
\times \{\pm\}. $$  

Hence, $G$ is a 
composition of computable maps and so is computable.  By
Part \ref{axiom:split} of the Proposition
\ref{prp_propertiesOfS},  there is an induced computable 
map $K: \mathcal{T} \to \mathcal{T}  
$ so that for each $T$, $K ({T})$ computes $G ^{T}: \mathbb{N}  \to 
B \times \{\pm \}$, $G ^{T} (n) = G (T,n). $  
By construction, if $T \in \mathcal{T} _{B}$ then 
$T' = (K (T) )' $. So that we are done.
\end{proof}
\subsubsection{Constructing decision Turing 
machines} 
\label{sec:Constructing decision Turing machines}
We will need a few
preliminaries as we need to deal with the following issue.
For a given computable $\mathbb{N} \to B \times
\{\pm\}$, we may construct a total computable map with the
same image, but possibly losing the stability conditions. So we
must adjust the totalization procedure to keep track of the stability conditions.

The following is well known:
\begin{lemma} \label{lem_total}
There is a computable map $Tot: \mathcal{T} \to \mathcal{T}
$ s.t. $\forall T \in \mathcal{T} $:
\begin{enumerate}
	\item $Tot (T)$ is total for each $T$.
	\item $\image (Tot (T)) = \image T$.
\end{enumerate}
\end{lemma}

Let 
\begin{align*}
& \pi _{\mathbb{N} }: \mathbb{N} \times
B \times \{\pm\} \to \mathbb{N}  \\
& \pi _{B \times \{\pm \} }: \mathbb{N} \times
B \times \{\pm\} \to B \times \{\pm \} 
\end{align*}
be the natural
projections.
We say that a subset $S \subset \mathbb{N} \times
B \times \{\pm\}$ is \emph{graphical}, if $\pi
_{\mathbb{N} } ^{-1} (n) $ is at most one element for each $n \in \mathbb{N} $.  

For a graphical  $S \subset \mathbb{N} \times B
\times \{\pm\}$ there is a natural total order, defined as
follows. For $s,s' \in S$, $s \leq s'$
if $\pi _{\mathbb{N} } (s) \leq \pi _{\mathbb{N} } (s')  $.
So a finite graphical $S \subset \mathbb{N} \times B
\times \{\pm\}$ 
determines a list $l _{S} \in L (B \times \{\pm
\})$ by
\begin{equation} \label{eq_l_s}
l _{S} (i) = \pi _{B \times \{\pm\}} (s
_{i}),
\end{equation}
where $s _{i}$ is the $i-1$ element of $S$ with
respect to the total order.

%


%
\begin{definition} \label{def:Tnstable}
Let $l \in L (B  \times \{\pm \} ) $. Define $b \in   
B$  to be \textbf{\emph{$l$-stable}}  if there is 
an $m \leq \length (l) $ s.t. $l (m) = (b,+)$ and there is no 
$m < k \leq  \length (l)$ s.t. $l (k) = (b,-)  $. 
%
\end{definition}
Given a map $f: \mathbb{N} \to B \times \{\pm \}$ 
let $\operatorname {gr} f: \mathbb{N} \to \mathbb{N} \times
B \times \{\pm \}$ be the map $\operatorname {gr}
f (n) = (n, f (n))$.
Expanding on Lemma \ref{lemma:TuringK}, there is a total
computable map:
\begin{equation} \label{eq_Gr}
Gr _{B}: \mathcal{T} \to \mathcal{T}, 
\end{equation}
with the property: $Gr _{B } (T)$ computes the map
$\operatorname {gr} (K _{B } (T)') $.

Set 
\begin{equation} \label{eq_Tot}
f ^{T} := e _{\mathbb{N} \times B \times \{\pm \}} ^{-1} \circ  Tot \circ
Gr _{B} (T) \circ e _{\mathbb{N}}.
\end{equation}
So that $f ^{T}: \mathbb{N} \to \mathbb{N} \times
B \times \{\pm \} $ and is total.

Finally, define $$G: B \times \mathcal{T} \times \mathbb{N} 
\to \{\pm \} $$ to be the map:
\begin{equation*}
   G (b, T, n)  = \begin{cases}
   +, &\text{$b$ is $l _{S}$-stable, for $S = \image (f ^{T}|
	 _{\{0, \ldots n\}})$}.  \\
      -, &\text{otherwise}. \\
\end{cases}
\end{equation*}
\begin{lemma} \label{lem_Gcomputable}
$G$ is computable.
\end{lemma}
\begin{proof} [Proof]
This is clear, but we make this explicit. 
Let 
\begin{equation}
   \label{eq:eqg}
  g: \mathbb{N} \to  L (\mathbb{N} ) 
\end{equation}
be the map $g (n) = \{0, \ldots, n\}$, it is clearly 
computable directly by part \ref{axiom:utility} of the
Proposition \ref{prp_propertiesOfS}.
Let $$L ^{graph} (\mathbb{N} \times B \times \{\pm
\}) \subset L (\mathbb{N} \times B \times \{\pm
\})  $$
consist of lists $l$ s.t. $S= \image l$ is graphical. The
latter subset is given the induced encoding by restricting
$e _{L (\mathbb{N} \times B \times \{\pm
\})}$.

Then we can express $G$ as the composition of the sequence 
of maps:
\begin{align*}
   & B \times \mathcal{T} \times \mathbb{N} 
    \xrightarrow{id \times (Gr _{B}) \times g} 
    B \times \mathcal{T}  \times L(\mathbb{N}) 
      \xrightarrow{id \times L(U) } B \times L 
     (\mathbb {N})  \\ & \xrightarrow{id
		 \times L(e
		 _{\mathbb{N} \times B 
     \times \{\pm \} } ^{-1})}  B \times
		 L ^{graph} (\mathbb{N} \times B 
     \times \{\pm \} ) \xrightarrow{id \times ord} B \times
		 L (B 
     \times \{\pm \} ) \to \{\pm\},
\end{align*}
where 
\begin{enumerate}
\item $ord: L ^{graph} (\mathbb{N} \times B \times \{\pm
\})
\to L (B \times \{\pm \})$ is the map  
$ord(l) = l _{S}$, for $S= \image l$. 
\item The last map is:
\begin{equation*}
   (b,l) \mapsto \begin{cases}
      +, &\text{ if $b$ is $l$-stable} \\
      -, & \text{ otherwise}.
   \end{cases}
\end{equation*}
\end{enumerate}
The latter two maps are computable by explicit verification. 
In particular all the maps in the composition are computable and so $G$ is computable.

\end{proof}

Let 
\begin{equation}
   \label{eq:DB}
   Dec _{B}: \mathcal{T} \to \mathcal{T},
\end{equation}
be the computable total map corresponding $G$ via Axiom 
\ref{axiom:split}, so that $Dec _{B} (T) $ is the 
Turing machine computing 
\begin{equation*}
G ^{T}: B \times \mathbb{N} \to \{\pm 
\}, \quad G ^{T} (b,n) = G (b,T,n). 
\end{equation*}
The following is immediate from the construction: 
\begin{lemma}
   \label{lemma:DecB}
$Dec _{B} (T) $ has the properties:
\begin{enumerate}
	\item $\forall T \in \mathcal{T} \;  Dec _{B} (T) \in
	\mathcal{D} _{B}$.
	\item $Dec _{B} (T) $ is total.
\end{enumerate}
\end{lemma}


\begin{definition}\label{def:Ddecided}
For a map $D: B \times N \to \{\pm \} $, we say that $b \in B$ is \textbf{\emph{$D$-decided}}
if there is an $m$ s.t. $D (b,m) =+$ and for all 
$n \geq m$  $D (b,n) \neq -$.  Likewise, for $T 
\in \mathcal{D} 
_{B}$ we say that  $b \in B$ is \textbf{\emph{$T$-decided}} if 
it is $T'$-decided. Also for $T \in 
\mathcal{T} _{B }$ we say that $b$ is 
\textbf{\emph{$T$-stable}}  if it is $T'$-stable in the
sense of Definition \ref{def:Mstable}. 
 \end{definition}

\begin{lemma} \label{lemma:convert} Suppose that $T \in 
\mathcal{T} _{B}$  then $b$ is $T$-stable iff $b$ is $Dec _{B} (T)$-decided. 
\end{lemma}
\begin{proof} 
Suppose that $b$ is $T$-stable. In particular, there is an $m \in \mathbb{N} $ so 
that $T (m) = (b,+)$ and there is no $n>m$ so that $T (m)
= (b,-)$.  By construction, there exists an $m'$ s.t. $f
^{T} (m') = (m, b,+)$. It is then immediate that $b$ is $l
_{S _{n}}$-stable, for any $S _{n} = \image (f ^{T}| _{\{0,
\ldots, n\}})$ s.t. $n \geq m'$. 

Then by construction of $G$, $G (b, T, m') =+$ and there is no
$n>m'$ s.t. $G (b, T, m') =-$.
Thus $b$ is $G ^{T}$-decided and so $Dec _{B} (T) $-decided.

The converse is also clear. Suppose $b$ is $G ^{T}$ decided.  
Let $N$ be s.t. 
\begin{enumerate}
	\item 
$G (b, T, N) = +$. \label{item_+}
\item  There is no $n>
N$ s.t. $G (b,T,n) = -$. \label{item_-}
\end{enumerate}
By property \ref{item_+} there is an $n_0
\in \{0, \ldots, N\}$ s.t. $f ^{T} (n_0) = (n _{b}, b,
+)$, and there is no $n \in \{0, \ldots, N\}$, with $pr
_{\mathbb{N} } \circ  f ^{T} (n) > n _{b}$, s.t. $pr
_{B \times \{\pm \}} \circ  f ^{T} (n) = (b,-)$.

Now, if $b$ is not $T$-stable then there is a $n'> n _{b}$ s.t.
$T (n') = (b,-)$. By the previous paragraph, $$n' \notin
\image pr _{\mathbb{N} } \circ  f ^{T}| _{\{0, \ldots,
N\}}.$$ So there is an $n > N$
so that $f ^{T} (n) = (n', b,-)$. But then $G (b,T,n) = -$
which contradicts property \ref{item_-}.
\end{proof}
\begin{example}
   \label{example:dio2} By the Example \ref{remark:dio}
	 above there is a computable map $$P=Dec _{Pol} (A): Pol \times \mathbb{N} \to \{\pm\} $$ that stably soundly decides if a Diophantine polynomial has integer roots, meaning:
\begin{equation*}
\text{$p$ is $P$-decided} \iff \text{$p$ has no integer roots}.
\end{equation*}
\end{example}


\begin{definition} \label{def:ThetaDT} Given a pair of maps 
$$M _{0}: {B} \times \mathbb{N} \to \{\pm\}  $$  $$M _{1}:  {B} \times \mathbb{N} \to 
\{\pm\},$$ we say that they are $ \textbf{\emph{stably
equivalent}}$ if
\begin{equation*} 
\text{$b$ is $M$-decided} \iff \text{$b$ is $M'$-decided}.
\end{equation*}
If $T \in \mathcal{D} _{B}$ then we say that $T$ 
\textbf{\emph{stable computes}} $M$ iff $T'$ is stably equivalent to $M$. 
\end{definition}
\subsection {Decision maps in first order theories} Let 
$\mathcal{L}$ be as in the Introduction. 
Let $\mathcal{T}_{\mathcal{L}}$ be as in 
Definition \ref{def:T_B} with respect to 
$B=\mathcal{L} $. 
The following is a version for stably c.e. theories of the classical fact, going back to  G\"odel, that for a 
theory with a c.e. set of axioms we may computably enumerate its theorems. Moreover, the procedure to obtain the corresponding
Turing machine is computably constructive. 
\begin{notation} \label{eq:notationTs}
Note that each $T \in \mathcal{T} _{\mathcal{L}}$,  
determines the set 
\begin{equation*}  
(T') ^{s} \subset \mathcal{L},
\end{equation*}
called the stabilization of $T'$, we hereby 
abbreviate the notation for this set as $T ^{s}$. 
\end{notation}

\begin{lemma} \label{lemma:CM}
Let $\mathcal{L} $ be given. There is a computable total map:
\begin{equation*}
C: \mathcal{T} \to \mathcal{T}
\end{equation*}
so that $\forall T \in \mathcal{T}: C (T) \in \mathcal{T}
_{\mathcal{L}} $.
If in addition $T \in \mathcal{T} _{\mathcal{L}} $ 
then $(C (T)) ^{s} $  is the deductive closure of $T ^{s}$.
%
\end{lemma}
 \begin{proof}
Let $L (\mathcal{L} ) $ be the list construction on
$\mathcal{L} $ as previously. Let $\Phi: L (\mathcal{L}
) \times \mathbb{N} \to \mathcal{L}$ be as in the Axiom
\ref{axiom_L}.

Using the map $ord$ from Lemma \ref{lem_Gcomputable}, define a map 
\begin{equation*}
   \zeta: L (\mathcal{L} ) \times L ^{graph} (\mathbb{N} \times \mathcal{L} \times \{\pm \})  \to \{\pm\} 
\end{equation*}
by 
\begin{equation*}
   \zeta (l,l')  = \begin{cases}
        +, &\text{if for each $0 \leq i \leq \length (l)$, 
        $l (i)$ is $ord (l')$-stable.} \\
      -, &\text{otherwise}. \\
\end{cases}
\end{equation*}
Utilizing Proposition \ref{prp_propertiesOfS} we readily see that $\zeta$ is computable.

Now define $H$ to be the composition of the 
sequence of maps:

\begin{equation*}
    \mathcal{T} \times  L (\mathbb{N} ) 
    \xrightarrow{Gr _{\mathcal{L} } \times id} \mathcal{T} \times L 
    (\mathbb {N}) \xrightarrow{LU} L (\mathbb {N})  
    \xrightarrow{L (e _{\mathbb{N}  \times \mathcal{L} \times \{\pm 
    \} } ^{-1}) } L ^{graph} (\mathbb{N} \times\mathcal{L} \times \{\pm \}).
\end{equation*}
All the maps in the composition are computable directly by
the Proposition \ref{prp_propertiesOfS} and Lemma
\ref{lem_inverse} and so $H$ is computable.

We may now construct our map $C$. In what follows $\cup$ will be the natural list union operation as previously in \eqref{eq:listunion}.
Set $$L _{n} (\mathbb{N} ):= \{l \in L (\mathbb{N} 
) | \max l \leq n, \mbox { $\max l$ the maximum of $l$ as a 
map} \}.$$
Let $$pr _{\mathcal{L} }: \mathbb{N}  \times \mathcal{L} \times \{\pm \} \to 
\mathcal{L} $$ be the natural projection.
For $n \in \mathbb{N} $, define $U ^{T} _{n} \in
L (\mathcal{L} \times \{\pm \} ) $ recursively 
by $U ^{T}_{0} := \emptyset$, 
$$U ^{T}_{n+1}:= U ^{T} _{n} \cup \bigcup
_{l \in L _{n+1} (\mathbb{N})} \bigcup _{0 \leq m \leq
n+1}(\Phi (L _{pr _{\mathcal{L}}} \circ H(T, 
l),  m), \zeta (L _{pr _{\mathcal{L}}} \circ  H (T,l), H (T, \{0, \ldots, n+1\})  ).$$

As in Example \ref{remark:dio} we define $$U ^{T}: 
\mathbb{N} \to \mathcal{L} \times  \{\pm \}, \quad U ^{T} (n) :=
U ^{T} _{n+1} (n).$$ 
And this induces a total map $$U: \mathcal{T} 
\times \mathbb{N} \to \mathcal{L} \times  \{\pm 
\}, 
$$ $U (T,n):=U ^{T} (n)  $.
$U$ is computable by explicit verification, 
utilizing Proposition \ref{prp_propertiesOfS}, and the 
recursive program for $\{U ^{T} _{n}\}$.  
Hence, by part \ref{axiom:split} of Proposition
\ref{prp_propertiesOfS}, there is an 
induced by $U$ computable total map:
\begin{equation*}
   C: \mathcal{T} \to \mathcal{T}, 
\end{equation*}
s.t. for each $T \in \mathcal{T} $, $C (T)  $ 
computes $U ^{T}$. 

By construction, this has the needed properties, and we are
done.
\end{proof}
\begin{notation}
\label{not_vdashalpha} Given $M: \mathbb{N} \to \mathcal{L}
\times \{\pm \}$, and some translation $i: \mathcal{A} \to
\mathcal{L} $, we use from now the shorthand: $M \vdash
^{i} \alpha$, for $M ^{s}
 \vdash ^{i} \alpha$, whenever $\alpha \in
\mathcal{A} $, and whenever this may cause no confusion.
\end{notation}

Let $\mathcal{F} _{0} $, as in the introduction, denote
the set of formulas $\phi$ of arithmetic with
one free variable so that for any term $n$, $\phi (n)$ is
$\Sigma ^{0} _{0}$ and in particular is $Q$-decidable.  
\begin{definition} \label{def:speculative}
Let $\mathcal{L} $ be a first order language, with some
translation $i: \mathcal{A} \to \mathcal{L} $. 
We say that $M: \mathbb{N} \to \mathcal{L} \times \{\mathbb{\pm}\} $ is \textbf{\emph{speculative}} (with respect
to $i$) if the
following holds. 
Let $\phi \in \mathcal{F} _{0} $, and  
set 
\begin{equation} \label{eq:alphaphi}
  \alpha _{\phi}  = \forall m \; \phi (m),
\end{equation}
then $$(\forall m \; Q \vdash \phi (m))  \implies M 
\vdash ^{i} \alpha _{\phi} .$$ 
\end{definition}

Note that we previously constructed an Example 
\ref{remark:dio} of a Turing machine, with an 
analogue of this speculative property. Moreover, we have the following crucial result, which to paraphrase 
states that there is an operation $Spec$  that 
converts a stably c.e. theory to
a speculative stably c.e. theory, at
a certain loss of consistency.
\begin{theorem} \label{lemma:speculative} 
Let $\mathcal{L} $ be given and $i: \mathcal{A} \to
\mathcal{L} $ be some fixed 2-translation. Then there is
a computable total map $Spec _{i}: \mathcal{T} \to \mathcal{T} $, with the following properties:   
\begin{enumerate}
\item  $\image Spec _{i} \subset \mathcal{T} _{\mathcal{L} } $.
\label{prop:0}
   \item  \label{prop:1}  Suppose that $T \in \mathcal{T} 
   _{\mathcal{L} }$ and $T ^{s} \vdash ^{i} Q$, set $T _{spec}=Spec _{i} (T) $ then
$T' _{spec}$ is speculative, moreover if $T'$ is 
total then so is $T' _{spec}$. 
   \item \label{prop:2}  Using Notation 
   \ref{eq:notationTs}, if $T \in \mathcal{T} 
   _{\mathcal{L}}$ then $T _{spec} ^{s} 
    \supset T  ^{s} $
      \item  \label{prop:3} If $T \in \mathcal{T} 
   _{\mathcal{L}}$ and $T ^{s}$ is 
      $1$-consistent with respect to $i$ then $T_{spec} ^{s} $ is consistent.
\end{enumerate}
\end{theorem}
\begin{proof} 
     $\mathcal{F} _{0}$ is assumed 
     to be encoded so that the map $$ev: 
     \mathcal{F} _{0} \times \mathbb{N}   \to 
     \mathcal{A}, \quad (\phi, m) 
     \mapsto \phi (m)$$ is computable. 
Let $G \subset \mathcal{F} _{0}$ be the subset of formulas
$\phi$ s.t. $\forall n \; Q \vdash \phi (n ^{\circ})$, where
$n ^{\circ }$ denotes the corresponding numeral. 

%
We then need:
\begin{lemma} \label{lemma:F} There is a total
   computable map $J: \mathbb{N} \to \mathcal{F} _{0}  \times 
   \{\pm\}$  with the property: $$J ^{s} = G. $$ 
\end{lemma}
\begin{proof}
The construction is analogous to the construction in the Example \ref{remark:dio} above.  
Fix any total, bijective, Turing machine $$Z:
   \mathbb{N} \to \mathcal{F} _{0}.$$
For a $\phi \in \mathcal{F} _{0} $ we will say that it is $n$-\textbf{\emph{decided}} if 
\begin{equation*}
   (\forall m \in \{0, \ldots, n\}) \; Q \vdash \phi (m).
\end{equation*}
In what follows each $J _{n} $ has the type of 
ordered finite list of elements of $\mathcal{F} 
_{0}\times \{\pm\}$, and $\cup$ will be the 
natural list union operation, as previously.
Define $\{J _{n}\} _{n \in \mathbb{N} }$ recursively by $J _{0} 
:=\emptyset$,
\begin{align*}
   {J} _{n+1}& := J _{n}  \cup  \bigcup _{\phi \in \{Z ({0}), \ldots, Z (n) \}}  (\phi, d ^{n} (\phi)),
    \\
    & \mbox{ where  $d ^{n} (\phi) = +$ if $\phi$ is $n$-decided and $d ^{n} (\phi) = - $ otherwise}.
 \end{align*}
We set $J (n) := J _{n+1} (n) $. This is a total 
map $$J: \mathbb{N} \to \mathcal{F} _{0} 
\times \{\pm \}, $$ having the property $F 
(\mathbb{N} ) = \cup _{n} \image (J _{n}) $.  $J$ is computable by explicit verification, using Proposition \ref{prp_propertiesOfS}.
\end{proof}
Returning to the proof of the theorem. Let $K=K _{\mathcal{L}}: \mathcal{T}
\to \mathcal{T} $ be as in Lemma 
\ref{lemma:TuringK}.  For $\phi \in \mathcal{F} 
_{0}$ let $\alpha _{\phi}$   be as in \eqref{eq:alphaphi}.
 Define: $H: \mathcal{T} \times \mathbb{N} \to 
 \mathcal{L} \times \{\pm \} $ by
\begin{equation*}
H (T,n): = \begin{cases} (K (T) )' 
(n), &\text{ if $n=2k+1$} \\
   (i  (\alpha_{pr _{\mathcal{F} _{0}} \circ J (n)}), pr
	 _{\pm} ({J (n)})), &  \text{ if $n=2k$},
\end{cases}  
\end{equation*}
   where $pr _{\mathcal{F} _{0}}: \mathcal{F} _{0} \times
   \{\pm\} \to \mathcal{F}$, and $pr _{\pm}:
   \mathcal{F} _{0} \times \{\pm\} \to \{\pm\}$ are the
   natural projections. 
$H$ is computable by Proposition \ref{prp_propertiesOfS}. (Factor $H$ as a composition of computable maps as previously.) 

Let $Spec _{i}: \mathcal{T} 
\to \mathcal{T} $ be the computable map 
corresponding to $H$ via Axiom \ref{axiom:split}.
In particular, for each $T \in \mathcal{T} $, 
$Spec _{i} (T) $ computes the map $$T' _{spec}:= H 
^{T}: \mathbb{N} \to \mathcal{L} \times \{\pm \}, 
\quad H ^{T} (n) = H (T,n), $$  which by construction is
speculative. Now, $Spec _{i} (T) $ satisfies the Properties \ref {prop:0}, \ref{prop:1}, \ref{prop:2}
immediately by construction. 

It only remains to check Property
\ref{prop:3}.
In what follows, let $T \in \mathcal{T} _{\mathcal{L} }$,
and set $F _{spec}
= (T'  _{spec}) ^{s}$, and $F= (T') ^{s}$, which is
1-consistent by assumption.
\begin{lemma}
   \label{lemma:Sconsistent}
$F _{spec}$ consistent unless for some $\phi \in G$ 
$$F  \vdash ^{i} \neg \forall m \; \phi (m).$$ 
\end{lemma}
\begin{proof} Suppose that $F _{spec}$ is inconsistent so that:
   $$F \cup \{i(\alpha _{\phi _{1}}), \ldots, i(\alpha 
   _{\phi _{n}})\} \vdash \alpha \land \neg \alpha$$ 
   for some $\alpha \in \mathcal {L}$, and some 
   $\phi _{1}, \ldots, \phi _{n} \in G$. Hence, 
   $$F \vdash ^{i} \neg (\alpha _{\phi _{1}} 
   \land \ldots \land \alpha 
   _{\phi _{n}}), \; \text{as $i$ preserves logical
	 operators.} $$  
   But $$\alpha _{\phi _{1}} 
   \land \ldots \land \alpha 
   _{\phi _{n}} \iff \forall m \; \phi (m), $$ where 
   $\phi$ is the formula with one free variable: 
   $\phi (m) :=   \phi _{1} (m)   \land 
   \ldots \land \phi _{n} (m)$.
   Clearly $\phi \in G$, since $\phi 
   _{i} \in G$, $i=1, \ldots, n$. Hence, the 
   conclusion follows.
\end{proof}
Suppose that $F _{spec}$ inconsistent, then by the lemma 
above for some $\phi \in G$: $$F \vdash ^{i} \exists m \; \neg \phi (m).$$ 
By assumption $F$ is
$1$-consistent with respect to $i$ and so: $$\exists m \;
F  \nvdash ^{i} \phi (m).$$ But $\phi$ is in 
$G$, and $F \vdash ^{i} Q $ (part of our 1-consistency
assumption) 
so that $\forall m \; F \vdash ^{i} \phi (m) $ and so
$$\exists m \; F \vdash \neg \phi (m) \land \phi 
(m).$$

So $F $ is inconsistent, a contradiction, so $F _{spec}$ is consistent.
\end{proof}

\section{The stable halting problem} \label{section:systemMpm} 
We now discuss a version of the halting problem in the
context of stable computability. This will illustrate the
main ideas that will be used in the proof of the
incompleteness theorems, and will allow us to setup notation for later use.

Let $\mathcal{D} _{\mathcal{T} } \subset \mathcal{T} $ be as in 
Definition \ref{def:calDB} with respect to $B= 
\mathcal{T} $. 

\begin{lemma} \label{lemma:TuringD} 
There is a computable total map $$\Omega: \mathcal{T} \to \mathcal{T},$$ with the 
properties:
\begin{enumerate}
	\item For each $T$, $\Omega   (T) \in \mathcal{D}
	_{\mathcal{T} }$. 
	\item If $T \in \mathcal{T} _{\mathcal{D} }$ then  $\Omega  
(T) $ and $T$ compute the same maps $\mathcal{T} \times
\mathbb{N} \to  \{\pm \} $, that is $T' = (\Omega (T))'$.   \label{property_2}
\end{enumerate}
\end{lemma}
\begin{proof} [Proof]
Analogous to the proof of Lemma \ref{lemma:TuringK}.
\end{proof}

\begin{definition}\label{def:Tdecided}
For $T \in \mathcal{D} _{\mathcal{T}}$, 
\textbf{\emph{$T$ is $T$-decided}}, is a special 
case of Definition \ref{def:Ddecided}. Or more 
specifically, it means that the element $T \in 
\mathcal{T} $ is $T'$-decided. 
We also say that $T$ is \emph{not}  $T$-decided, when $\neg 
(\text{$T$ is $T$-decided})$ holds.
\end{definition}
%
%
In what
follows, denote by $s (T)$ the sentence:
\begin{equation*}
T \text{ is not $\Omega (T)$-decided}.
\end{equation*}
which can be naturally interpreted as a sentence of
arithmetic, but we leave this implicit for the moment.
\begin{definition} 
We say that $D: \mathcal{T} \times \mathbb{N} \to \{\pm\}$
is \textbf{\emph{stably sound}}  if
\begin{equation*}
\forall T \in \mathcal{T}  \; (\text{$T$ is $D$-decided}) \implies s (T) .
\end{equation*}
We say that $D$ \textbf{\emph{stably decides $T$}} if: 
\begin{equation*}
   s (T) \implies \text{$T$ is $D$-decided}.
\end{equation*}
We say that $D$ is \textbf{\emph{stably sound and complete}}
if $D$ is stably sound and $D$ stably decides $T$ for all $T \in \mathcal{T}.$
\end{definition}

The informal interpretation of the above is that 
each such $D$ is understood as an operation with the properties:
\begin{itemize}
\item  For each $T,n$ $D (T,n) = +$ if and only if 
$D$ ``decides'' the sentence $s (T)$ is true, at the moment
$n$. 
\item For each $T,n$ $D (T,n) = -$ if and only if $D$ cannot 
``decide'' the sentence $s (T)$ at the moment $n$, or $D$
``decides'' that $s (T)$ is false.
  \end{itemize} 

In what follows for $T \in \mathcal{T}$, and $D$ 
as above, $\Theta _{D,T} $ is shorthand for 
the sentence: $$\text{$T$ stably computes $D$},$$ 
where this is as in Definition \ref{def:ThetaDT}.
\begin{lemma} \label{lemma:Dstablysound} If $D$ is stably sound then
\begin{equation*}
  (\forall T \in \mathcal{T})  \;  \neg \Theta  _{{D}, T}  \lor  \neg (\text{$T$ is $D$-decided}).
\end{equation*}
\end{lemma}
\begin{proof} 
If $\text{$T$ is $D$-decided}$ then 
since $D$ is stably sound, \text{$T$ is not $\Omega (T)$-decided}.
So if in addition $\Theta _{D,T}$ then by property
\ref{property_2} of $\Omega$, $(\Omega (T))' = D$. And so 
$T$ is not $D$-decided a contradiction.
\end{proof}
The following is the ``stable'' analogue of 
Turing's halting theorem. 
\begin{theorem} \label{thm:decideP} There is no 
(stably) computable $D: \mathcal{T} \times \mathbb{N} \to
\{\pm \}$ that is stably sound and complete. 
\end{theorem}
\begin{proof} Let $D$ be 
stably sound and complete. Then by the above lemma we obtain:
\begin{equation} \label{eq:first}
(\forall T \in {\mathcal{T}}) \;  (\Theta  
_{{D}, T}  \vdash  \neg (\text{$T$ is $D$-decided})).
\end{equation}
Again as a consequence of property
\ref{property_2} of $\Omega$,  it is immediate: 
\begin{equation} \label{eq:second}
(\forall T \in {\mathcal{T}})  \; (\Theta  _{{D}, T} 
\implies \left( \neg (\text{$T$ is $D$-decided}))
\implies \neg (\text{$T$ is $\Omega (T)$-decided}) \right).
\end{equation}
So combining \eqref{eq:first}, \eqref{eq:second} above we obtain 
$$(\forall T \in {\mathcal{T}}) \; (\Theta _{{D},
T} \implies \neg (\text{$T$ is $\Omega (T)$-decided})). $$
But $D$ is complete so $$(\text{$T$ is 
$\Omega (T)$-decided}) \implies \text{$T$ is $D$-decided} $$
and so:
$$(\forall T \in {\mathcal{T}}) \; (\Theta _{{D}, T}
\implies (\text{$T$ is $D$-decided})). $$
Combining with \eqref{eq:first} 
we get $$(\forall T \in {\mathcal{T}}) \; \neg \Theta _{{D}, T},$$ which is what we wanted to prove.
\end{proof}
\section{Incompleteness theorems} \label{section:stablyconsistent} 
Let $s: \mathcal{T} \to \mathcal{A} $, $T \mapsto 
s (T) $ be as in the previous section. 
Define 
\begin{equation*}
   \label{eq:wtD}
   H: \mathcal{T} \times \mathcal{T} \times 
\mathbb{N} \to \{\mathbb{\pm}\},
\end{equation*}
by $$H (F,T,n):= (Dec 
_{\mathcal{L}} ({C(Spec _{i}(F))))'} (i \circ s (T),n ).$$
We can express $H$ as the composition of the sequence of maps:
\begin{equation} \label{eq:wtDH}
\mathcal{T} \times \mathcal{T} \times 
\mathbb{N} \xrightarrow{(Dec _{\mathcal{L}} \circ C 
\circ Spec _{i}) \times 
(i \circ s) \times id} \mathcal{T} \times \mathcal{L} \times 
\mathbb{N} \xrightarrow{id \times e _{\mathcal{L} 
\times \mathbb{N} }} \mathcal{T} \times 
\mathbb {N} \xrightarrow{U} \mathbb {N} 
\xrightarrow{e _{\{\pm \} } ^{-1}} \{\pm\}.
\end{equation}

Thus, $H$ is a 
composition of maps that are computable by Proposition
\ref{prp_propertiesOfS} and so $H$ is computable.  Hence, by
part \ref{axiom:split} of Proposition
\ref{prp_propertiesOfS}, there 
is an associated total computable map:
\begin{equation} \label{eq:Tur}
  Tur: \mathcal{T} \to \mathcal{T},  
\end{equation}
s.t. for each $F \in \mathcal{T} $, $Tur (F) $ 
computes the map 
\begin{equation}\label{eq_DK}
  D ^{F}: \mathcal{T} \times 
\mathbb{N} \to \{\pm \}, \; D ^{F} (T,n) = H 
(F,T,n), 
\end{equation}

In what follows, $T \in \mathcal{T} _{\mathcal{L} }$,
and we rename $T'$ as $M$ so $M: \mathbb{N} \to 
\mathcal{L} \times \{\pm \} $. 

\begin{notation}
\label{not_} As usual, for a map $J: \mathbb{N} \to 
\mathcal{L} \times \{\pm \}$, notation of the
form $J \vdash \alpha$ means $J ^{s} \vdash \alpha$. If $i:
\mathcal{A} \to \mathcal{L} $ is a translation, then $J
\vdash ^{i} \alpha $ will denote $J ^{s}\vdash ^{i} \alpha
$. $J$ is $(n)$-consistent stands for $J ^{s}$ is $(n)$-consistent.
\end{notation}

\begin{proposition}   \label{thm:synctatic}
For $(M,T)$ as above and given a 2-translation $i:
\mathcal{A} \to \mathcal{L} $ we have:
\begin{enumerate}
	\item 
$\text{$M$ is 1-consistent relative to the
translation $i$} \implies M  
\nvdash i \circ s (Tur (T)).$
	\item $\text{$M$ is 2-consistent relative to the
	 translation $i$} \implies M  
   \nvdash \neg i \circ s (Tur (T)).  
$
\item  $\text{$M$ is 1-consistent relative to the
translation $i$} \implies s (Tur (T)).$
\end{enumerate}
\end{proposition}
\begin{proof} 
Since $T$ will be fixed, let us abbreviate $s (Tur (T))$ by $\mathfrak s$.
Set $N:= (C \circ Spec _{i} (T))' $, in particular this is 
a speculative (with respect to $i$) map $\mathbb{N} \to 
\mathcal{L} \times \{\pm \}  $.
Suppose that $M  
\vdash i (\mathfrak s)$. In particular, $i (\mathfrak s)$ is $N$-stable, and so by Lemma 
\ref{lemma:convert}
$i (\mathfrak s) $ is $Dec _{\mathcal{L}} (C(Spec _{i}
(T)))$-decided. And so $Tur (T)$ is ${D} ^{T}$-decided by definitions, where $D ^{T}$ is as in \eqref{eq_DK}.

Now, since $Tur (T) $ computes $D ^{T}$ by construction, we
have $(\Omega (Tur (T)))' = D ^{T}$ and so:

\begin{align} \label{eq_OmegaTurT}
\begin{split}
	\text{$Tur (T)$ is ${D} ^{T}$-decided} & \iff \text{$Tur
 (T)$ is $\Omega (Tur (T))$-decided} \\
 & \iff \neg \mathfrak s.
\end{split}
\end{align}
That it to say:
\begin{equation} \label{eq:MimpliesetaZ}
     (M  \vdash i (\mathfrak s)) \implies \neg \mathfrak s.
\end{equation}

Now, $\neg \mathfrak s$ is an arithmetic sentence of the form: 
\begin{equation} \label{eq_sentenceS}
   \exists m \forall n \; \gamma (m,n),
\end{equation}
where  $\gamma  \in \Sigma ^{0} _{0}$. 
And we have:
\begin{align}
   \neg \mathfrak s  & \implies (\exists m \forall n) \; 
   Q \vdash \gamma (m,n), \quad \text{since $\gamma$ is
	 $\Sigma ^{0} _{0}$}  \\
   & \implies \exists m \; N   \vdash ^{i}  \forall 
   n \; \gamma  (m,n),
   \quad \text{since $N$ is speculative} \\
   & \implies N  \vdash ^{i} \neg \mathfrak s, \; (\text{by
	 existential introduction}). 
\end{align}
And so combining with \eqref{eq:MimpliesetaZ}, we get:
$$(M \vdash i (\mathfrak s) ) \implies (N \vdash
i (\mathfrak s)) \land (N \vdash \neg i (\mathfrak s)).$$

Since by Theorem \ref{lemma:speculative} 
$$\text{$M $ is 1-consistent} \implies 
\text{$N $ is consistent},$$ 
it follows:
\begin{align} \label{eq:firstpart}
\text{$M $ is 1-consistent} & \implies  M  
\nvdash ^{i} \mathfrak s  \\
& \implies  (\text{$Tur(T)$ is not $\Omega
(Tur(T))$-decided}), \; \text{by \eqref{eq_OmegaTurT}} \\
& \implies \mathfrak s.
\end{align}
This proves the first and the third part of the proposition.

Now suppose $$(\text{$M $ is 2-consistent})
\land (M  \vdash ^{i} \neg \mathfrak s). $$ 

Set $$\phi (m)=\forall n \; \gamma (m,n),$$ where
$\gamma  (m, n)$ is as in \eqref{eq_sentenceS}.
Now, 
\begin{align*}
   M \vdash ^{i} \neg \mathfrak s & \iff M  \vdash ^{i}  \exists m \; \phi (m)\\
	 &  \implies  \exists m \; M   \nvdash ^{i}
	 \neg \phi (m), \quad \text{by $2$-consistency} \\
	 &  \implies \exists m \forall n \; Q \vdash \gamma (m,n),
	 \quad \text{as $M \vdash ^{i} Q$ and 
$\gamma (m,n) $ is $Q$-decidable}.
\end{align*}
And so,
   $$(\text{$M$ is 2-consistent}  \land (M  
   \vdash ^{i} \neg \mathfrak s )) \implies \neg \mathfrak s.$$ 
Now,
\begin{equation*}
   \neg \mathfrak s \implies N \vdash \mathfrak s, 
\end{equation*}
by definitions and \eqref{eq_OmegaTurT}. So:
\begin{align*}
   (\text{$M$ is 2-consistent})  \land (M  
\vdash ^{i} \neg \mathfrak s) & \implies N \vdash \mathfrak
s \\
& \implies \text{ $N $ is inconsistent } \\
& \implies \text{ $M $ is not 1-consistent, 
by Theorem \ref{lemma:speculative} } \\
   & \implies \text{ $M $ is not 
   2-consistent}. \\
\end{align*}
So we get a contradiction, and so:
$$ \text{$M$ is 2-consistent} \implies M \nvdash ^{i} \neg
\mathfrak s. $$
This finishes the proof of the proposition.
\end{proof}

\begin{proof} [Proof of Theorem \ref{thm_mainintro}] 
The computable map $\mathcal{G}$ is defined to be $T
\mapsto s (Tur (T))$.  Then the theorem follows immediately by the proposition above. 
\end{proof}	
	

\begin{proof} [Proof of Corollary \ref{cor_secondinc}]
Suppose that $$(F \vdash ^{i} \text{$F$ is 1-consistent})  \land
(F
\vdash ^{i} \text{$T$ stably enumerates $F$}) \land
(\text{$F$ is 1-consistent}) $$ then since $F \vdash ^{i} ZFC$, 
$$(F \vdash s (Tur (T))) \land \text{$F$ is
1-consistent}$$ by part three of Proposition \ref{thm:synctatic}. But
this contradicts part one of Proposition \ref{thm:synctatic}.
\end{proof}

\begin{proof} [Proof of Theorem \ref{thm:corollarySecondIncompleteness}]
Let $M$ be a standard model of $F _{i,
\mathcal{Z} }$ as in the definition of strong consistency.
In particular:
\begin{equation} \label{eq_alphaU}
(\forall \alpha  \in {F} _{i, \mathcal{Z}} ) \; M
\models \alpha.
\end{equation}

Suppose that $F \vdash ^{i} \text{$F$ is
1-consistent}$, then $M \models (F \vdash ^{i} \text{$F$ is
1-consistent})$, using that $M$ is standard.
Suppose also:
$$F \vdash ^{i} (\exists {T} \in \mathcal{T})  \; \text{$T$ stably computes $F$},$$ 
then by \eqref{eq_alphaU} $$M \models (\exists {T} \in \mathcal{T})  \; \text{$T$
stably computes $F$}.$$
And so we obtain:
\begin{equation*}
M \models (\text{$F$ is 1-consistent}) \land
((\exists {T} \in \mathcal {T}) \; \text{$T$
stably computes $F$}) \land (F \vdash ^{i} \text{$F$ is
1-consistent}),
\end{equation*}
that is
\begin{equation} \label{eq_models1}
M \models (\exists {T} \in \mathcal {T}
_{\mathcal{\mathcal{L}}})  \; (\text{$(T') ^{s}$ is 1-consistent}) \land
 (T' \vdash ^{i} \text{$(T') ^{s}$ is 1-consistent})).
\end{equation}
Rephrasing Corollary \ref{cor_secondinc} we get:
$$ZFC \vdash (\forall T \in \mathcal{T} _{\mathcal{L} })
\; (\text{$(T') ^{s}$ is
1-consistent} \implies  \text{$T' \nvdash ^{i} (T') ^{s}$ is
1-consistent}). $$
And $F \vdash ZFC$ and so again by \eqref{eq_alphaU} we get:
$$M \models \neg ((\exists {T} \in \mathcal{T} _{\mathcal{L}
}) \; (\text{$(T') ^{s}$ is 1-consistent}
\land T' \vdash ^{i} \text{$(T') ^{s}$ is
1-consistent})),$$
but this contradicts \eqref{eq_models1}.
\end{proof}

\appendix 
\section{Stable computability and physics - G\"odel's disjunction and Penrose} \label{section:Penrose}
We now give some partly physical motivation for the theory
above, in particular explaining why stable computability and
abstract languages were important for us.  We aim to be very
brief, as this is an excursion. But developing this appendix would be very interesting in an appropriate venue. 


We may say that a physical process is \emph{absolutely not
Turing computable}, if it is not Turing computable in any
``sufficiently physically accurate'' mathematical model. For
example, it is well known (see
for instance ~\cite{cite_EulerTuringComplete}) that solutions of fluid flow and $N$-body problems are generally non Turing 
computable (over $\mathbb{Z}$, and probably over 
$\mathbb{R}$ cf. \cite{cite_BlumShubSmalen}) 
as modeled in mathematics of classical 
mechanics.  But in a more physically accurate and 
fundamental model both of the processes above may become
computable.

The question posed by
Turing~\cite{cite_TuringComputingMachines},  but also by
G\"odel~\cite[310]{cite_Godel}  and more recently and much
more expansively by Penrose ~\cite{cite_PenroseShadows},
~\cite{cite_PenroseBeyondShadow}, ~\cite{cite_PenroseRoad}
is: 
\begin{question} \label{quest1}
   Are there absolutely not Turing computable physical
	 processes?  And moreover, are brain processes
	 absolutely not Turing computable?
\end{question}


\subsubsection {G\"odel's disjunction}
G\"odel argued for a 'yes' answer to Question 1, see
\cite[pg. 310]{cite_Godel}, relating the question to existence
of absolutely unsolvable Diophantine problems, see also
Feferman~\cite{cite_Feferman2006-SFEATA}, and
Koellner~\cite{cite_Koellner2018-KOEOTQ-3}, \cite{cite_KoellnerII2018-KOEOTQ-4} for a discussion.  

We now discuss the question from  the perspective of our
main results.
First by an idealized mathematician, we mean a mapping $H: \mathbb{N} \to
\mathcal{\mathcal{L}}
\times \{\pm \}$, with $\mathcal{L}$ a first order language.
The language $\mathcal{L} $ is meant to be chosen so that it
is sufficient to formalize physical laws (we don't need to
formalize everything, the context will be self apparent). This might
be the language of set theory, but perhaps one needs more.
\footnote {It is of course possible that no such language
exists, but this is not our expectation: mathematics is
unreasonably effective at formalizing the universe.}
The mapping $H$ is meant to be the actual time
stamped output of a mathematician, idealized so that their
brain does not deteriorate in time, see also Remark
\ref{remark_intro}.

We set $\mathcal{H} = H ^{s}$. Now, soundness and in particular
strong consistency (Definition \ref{def_stronglyconsistent}) of
the stabilization $\mathcal{H} $ is not an unreasonable
hypothesis for our ``idealized'' mathematician,
as mathematical knowledge does appear to stabilize on truth.
Here the word
`stabilize' is used in the standard English language sense, but in this setting this is equivalent to the soundness of the mathematical stabilization $H ^{s}$.


Without delving deeply into interpretations, we
suppose the following axioms for $\mathcal{H} $. 
\begin{enumerate}
	\item $\mathcal{H} $ is $\mathcal{L} $-definable. Meaning,
	that it is definable with respect to some $\mathcal{L} $-structure
	(informally, the latter is obtained from some mathematical (sub)model
	of the physical universe). \footnote {This is
	natural: as $H$ is determined by some physical processes, 
	it is $\mathcal{L} $-definable,
	since by assumption the physical laws are
	stated in the language $\mathcal{L}$.}
	\item $\mathcal{H} \vdash ^{i} ZFC $, where $i:
	\mathcal{Z} \to \mathcal{L}$ is
	a 2-translation, recall the
	Definition \ref{def_interprets}. 
	\item The ``Penrose
	property'' holds: $$\mathcal{H} \vdash ^{i}   (\mathcal{H}
	 \text{ is 1-consistent}.)  $$
\end{enumerate}
The ``Penrose property'' is motivated by ideas of
Roger Penrose as appearing in references above. This
property makes sense
if our idealized mathematician knows the definition of
$\mathcal{H} $, and asserts their soundness,
and hence 1-consistency of $\mathcal{H} $.  
That the definition of $\mathcal{H} $ is known, is not
unlikely.  The idea for this is to map the brain (synapses,
and other relevant fine structure); then assuming one knows the working of all underlying
physical processes, use this to reconstruct the $\mathcal{L}
$-theoretic definition of
$\mathcal{H} $. This would be a fantastically difficult
thing to do, but theoretically possible.  Given this, there
is no obvious reason to reject the above axioms.

The following just paraphrases Theorem
\ref{thm:corollarySecondIncompleteness}. 
\begin{theorem} \label{thm:1} For any $\mathcal{H}, \mathcal{L} 
$ as above, 
one of the following holds:
\begin{enumerate}
	\item $\mathcal{H} $ is not strongly consistent.	
	\item $\mathcal{H} \nvdash ^{i} \text{$\mathcal{H}$ is
	stably c.e.}$. (In particular,  given our interpretation
	of $\mathcal{H} $ as representing an idealized
	mathematician, it is unable to disprove existence of absolutely non Turing computable physical processes.)
\end{enumerate}
\end{theorem}

%
\subsection*{Acknowledgements} Peter Koellner for helpful discussions on related topics. 
\bibliographystyle{siam} 
\bibliography{link.bib}
\end{document}